\documentclass[12pt]{amsart}
\usepackage{amscd,amssymb}
\usepackage{graphicx,enumerate}
\usepackage{upref}
\usepackage{amsmath}
\usepackage{amstext}
\usepackage{amsthm}

\usepackage[arrow,matrix,arrow,color,line,curve,frame]{xy}
\usepackage{pgf, tikz}
\usepackage{url}
\usepackage{enumerate}

\makeatletter
\def\url@leostyle{%
  \@ifundefined{selectfont}{\def\UrlFont{\sf}}{\def\UrlFont{\small\ttfamily}}}
\makeatother
\usetikzlibrary{matrix,arrows}

\usepackage[colorinlistoftodos,bordercolor=red,backgroundcolor=red!20,linecolor=red,textsize=scriptsize]{todonotes}

\let\oldlabel=\label

\def\prellabel{\marginparsep=1em
    \def\label##1{\oldlabel{##1}\ifmmode\else\ifinner\else
         \marginpar{{\footnotesize\ \\ \tt
                    ##1}}\fi\fi}}

\def\inte{\operatorname{int}}
\def\conv{\operatorname{conv}}

\def\rank{\operatorname{rank}}
\def\Im{\operatorname{Im}}
\def\Hilb{\operatorname{Hilb}}
\def\Um{\operatorname{Um}}
\def\GL{\operatorname{GL}}

\def\E{\operatorname{E}}
\def\Aff{\operatorname{Aff}}

\def\height{\operatorname{ht}}
\def\sn{\operatorname{sn}}
\def\n{\operatorname{n}}
\def\gp{\operatorname{gp}}
\def\t{\operatorname{t}}

\def\be{\mathbf e}
\def\bf{\mathbf f}
\def\bg{\mathbf g}

\def\ba{\mathbf a}
\def\bb{\mathbf b}
\def\bh{\mathbf h}
\def\bl{\mathbf l}

\def\RR{{\mathbb R}}

\def\QQ{{\mathbb Q}}
\def\ZZ{{\mathbb Z}}
\def\NN{{\mathbb N}}

\def\FF{{\mathbb F}}

\def\kk{{\mathsf k}}

\let\phi=\varphi
\let\theta=\vartheta
\let\epsilon=\varepsilon

\advance\headheight1.15pt

\newtheorem{lemma}{Lemma}[section]
\newtheorem{corollary}[lemma]{Corollary}
\newtheorem{theorem}[lemma]{Theorem}
\newtheorem{proposition}[lemma]{Proposition}

\theoremstyle{definition}
\newtheorem{definition}[lemma]{Definition}
\newtheorem{remark}[lemma]{Remark}

\begin{document}

\title[Unimodular rows over monoid rings]{Unimodular rows over monoid rings}

\begin{abstract}
For a commutative Noetherian ring $R$ of dimension $d$ and a commutative cancellative monoid $M$, the elementary action on unimodular $n$-rows over the monoid ring $R[M]$ is transitive for $n\ge\max(d+2,3)$. The starting point is the case of polynomial rings, considered by A. Suslin in the 1970s. The main result completes a project, initiated in the early 1990s, and suggests a new direction in the study of $K$-theory of monoid rings. 
\end{abstract}

\author{Joseph Gubeladze}

\address{Department of Mathematics\\
         San Francisco State University\\
         1600 Holloway Ave.\\
         San Francisco, CA 94132, USA}
\email{soso@sfsu.edu}

\thanks{}



\maketitle

\section{Introduction}\label{Intro}

An $n$-row $\ba=(a_1,\ldots,a_n)$ with entries in a commutative ring $R$ is called \emph{unimodular} if $Ra_1+\cdots+Ra_n=R$. If $R$ is Noetherian then, using prime avoidance, one can pass from $\ba$ by elementary transformations to a unimodular row $\bb=(b_1,\ldots,b_n)$, such that the height of the ideal $Rb_1+\cdots+Rb_i\subset R$ is at least $i$ for $i=1,\ldots,n$; see Section \ref{Unimodular-rows}. Here an `elementary transformation' means adding a multiple of a component to another component. In particular, if the (Krull) dimension of $R$ is $d$ and $n\ge d+2$ then every unimodular $n$-row over $R$ can be reduced, by elementary transformations, to $(1,0,\ldots,0)$. This is the basis of the classical \emph{Serre Splitting} and \emph{Bass Cancellation Theorems} \cite[Ch. 4]{Bass}. Our main result is

\begin{theorem}\label{main}
Let $R$ be a commutative Noetherian ring of dimension $d$ and $M$ be a commutative cancellative (not necessarily torsion free) monoid. Then the elementary action on unimodular $n$-rows over $R[M]$ is transitive for $n\ge\max(d+2,3)$.
\end{theorem}

Observe that the class of monoids in Theorem \ref{main} is exactly the class of submonoids of arbitrary abelian groups.

The starting point is Suslin's result in \cite{Suslin-Linear} that, for $(R,d)$ as above and arbitrary $r\in\NN$, the elementary action on the set of unimodular $n$-rows over the polynomial ring $R[t_1,\ldots,t_r]$ is transitive whenever $n\ge\max(d+2,3)$. The polynomial ring $R[t_1,\ldots,t_r]$ is the monoid ring, corresponding to a free commutative monoid of rank $r$. Theorem \ref{main} completes the project, initiated in \cite{Elrows,Elrows2}, where the transitivity was shown for a restricted class of monoids. Namely, in the torsion free case, Theorem \ref{main} extends the class of monoids in \cite{Elrows,Elrows2} in the same way as the general convex polytopes extend the class of \emph{stacked} polytopes \cite{Stacked} (called the `polytopes of simplicial growth\rq{} in \cite{Elrows2}); see Section \ref{Phi} for the relationship between monoids and polytopes. The class of stacked polytopes is negligibly small within general convex polytopes. But also, monoids with torsion have not been considered before.

We now describe consequences and research directions, suggested by Theorem \ref{main}.

Theorem \ref{main} is new already when $R$ is a field, in which case it implies that $K_1(R[M])$ is described by \emph{Mennicke symbols} \cite[Ch. 6]{Bass}. One would like to know what properties of $K_1(R[M])$ can be inferred from this fact -- as the works \cite{Gubel-Nontrivial,India} show, $K_1(R[M])$ exhibits many interesting phenomena.

Using the \emph{Quillen induction} for projective modules \cite[Ch. V.3]{Lam}, one easily deduces from Theorem \ref{main} that, for $R$ and $M$ as in the statement and $M$ torsion free without nontrivial units, every finitely generated projective $R[M]$-module of rank greater than $d$, which is stably extended from $R$, is in fact extended from $R$. For regular $R$ and seminormal $M$ this is shown in \cite[Corolary 1.4]{Swan}, and for a Dedekind ring $R$ and a torsion free monoid $M$ this follows from \cite[Theorem 1.5]{Swan}. It is very likely that the techniques of \cite{Lindel,Rao} can be combined with the proof of Theorem \ref{main} to yield the full blown $K_0$-part of \cite[Conjecture 2.4]{HHA}, claiming that finitely generated projective $R[M]$-modules of rank $\ge\max(d+1,2)$ are cancellative and split off free summands. This would extend \cite{Lindel,Rao} in the same way as Theorem \ref{main} extends Suslin's mentioned result. It is also worth mentioning that a new proof of Anderson\rq{}s conjecture \cite{Anderson} can be derived from Theorem \ref{main}, in combination with \cite{K2}, following the outline in \cite[Exercises 8.7, 8.8, 9.4]{Kripo}. Ideologically, this is the same approach, though.

Theorem \ref{main} is proved by developing a unimodular row version of the polytopal/algebraic framework, which we call the \emph{pyramidal descent}. Three versions of this techniques were developed for projective modules \cite{Anderson}, classical $K$-groups \cite{K1Gubel,K2}, and higher stable $K$-groups \cite{Gubel-Nil}. However, putting the elementary rows in the same framework has resisted the attempts so far -- only initial steps were accomplished in \cite{Elrows,Elrows2}. The main obstruction for proving the general transitivity of the elementary action had been the non-existence of special non-monomial endomorphisms in monoid rings for producing monoid ring counterparts of multivariate monic polynomials. Here we find a new approach, allowing to circumvent this difficulty by suitably lifting critical steps to covering polynomial rings, where there is a ubiquity of endomorphisms, and even drop the torsion freeness assumption for monoids. The unimodular row version of the pyramidal descent is different from the previous versions in that no lifting to polynomial rings was used before. More importanty, it also suggests a possible approach to \cite[Conjecture 2.4]{HHA} on stabilizations of all higher $K$-groups of $R[M]$: one could try to show, based on \cite{Suslin-Stability,vdKallen}, that the mentioned stabilizations are no worse than those for $R[t_1,\ldots,t_r]$, $r\in\NN$. This would reduce the general case to the free commutative monoids. At present, the appropriate techniques for polynomial rings only exists for projective modules (unstable $K_0$) \cite{Lindel,Rao}, invertible matrices (unstable $K_1$) \cite{Suslin-Linear}, and unstable Milnor groups \cite{Tulen}. Of course, Theorem \ref{main} already implies the conjectured surjective $K_1$-stabilization.

Finally, by combining the techniques in Section \ref{reduction} and \cite[\S15]{Swan}, one easily extends Theorem \ref{main} to rings of the form $R[M]/I$, where $I$ is generated by a subset of $M$. Also, if $R$ is a field, the proof of Theorem \ref{main} can be converted, along the lines in \cite{Algorithm}, into an algorithm that,  for a given unimodular row over $R[M]$, finds elementary transformations, leading to $(1,0,\ldots,0)$.

\medskip A word on the organization of the paper: in Sections \ref{Preliminaries}--\ref{Reduction} we overview previous results and make series of reductions in the general case; Section \ref{Quasimonic} allows involvement of the general coefficient rings, without which Section \ref{Pyramidal-descent} would lead to the special case of Theorem \ref{main} when $R$ is field; in Section \ref{Torsion} we explain how to involve the monoids with torsion.

\medskip\noindent\emph{Notation:} $\NN=\{1,2,\ldots\}$ and $\gg$ is for `sufficiently larger than'.

\medskip\noindent\emph{Acknowledgment.} I am very grateful to the referee for (i) detecting a mistake in the concluding inductive step at the end of Section \ref{Global-pyramidal} and suggesting a correction (see Remark \ref{Old-complexity}), and (ii) suggesting the current form of Lemma \ref{avoidance}.

\section{Unimodular rows, monoids, polytopes}\label{Preliminaries}

\subsection{Unimodular rows}\label{Unimodular-rows} All information we need on unimodular rows (except the case $l>0$ of Theorem \ref{Suslin-main} below), including a detailed exposition of \cite{Suslin-Linear}, is in \cite{Lam}.

All our rings are commutative and with unit.  The \emph{elementary subgroup} $\E_n(A)\subset\GL_n(A)$ of the general linear group over $A$ is generated by the elementary $n\times n$-matrices, i.e., the matrices which differ from the identity matrix in at most one nonzero off-diagonal entry. The following two statements are, respectively, \cite[Corollary 1.4]{Suslin-Linear} and \cite[Theorem 7.2]{Suslin-Linear}:

\begin{lemma}\label{Normal}
For any ring $A$ and $n\ge3$, $\E_n(A)$ is normal in $\GL_n(A)$. 
\end{lemma}

\begin{theorem}\label{Suslin-main}
Assume $R$ is a Noetherian ring of dimension $d$ and $k,l,n\in\NN$ with $n\ge\max(d+2,3)$. Then the elementary action on unimodular $n$-rows over the Laurant polynomial ring $R[t_1,\ldots,t_k,s_1^{\pm1},\ldots,s_l^{\pm1}]$ is transitive.
\end{theorem}

For a ring $A$, the set of unimodular rows of length $n$ will be denoted by $\Um_n(A)$. The group $\E_n(A)$ acts on $\Um_n(A)$ by multiplication on the right. More precisely, the right multiplication of an elementary matrix $I_n+aE_{ij}$ corresponds to adding the $a$-multiple of the $i$-th component to the $j$-th component.

For two elements $\bf,\bg\in\Um_n(A)$ and a subring $S\subset A$, we write $\bf\underset{S}\sim\bg$ if $\bf$ and $\bg$ are in a same orbit of the $\E_n(S)$-action. The standard row $(1,0,\ldots,0)$ will be denoted by $\be$. In particular, $\bf\underset{S}\sim\be$ means that there a matrix $\epsilon\in \E_n(S)$, whose first row equals $\bf$.

In order to avoid confusion between rows and ideals, the ideal in $A$, generated by elements $a_1,\ldots,a_n$, will be denoted by $Aa_1+\cdots+Aa_n$.

For an ideal $I$ in a ring $A$, its \emph{height} $\height I=\height_AI$ is the minimal height of a prime ideal over $I$. The height of the unit ideal is set to be $\infty$. For an ideal $J\subset A[t]$, the ideal of leading coefficients of elements in $J$ will be denoted by $L(I)$.

\begin{lemma}\label{Lam-height}
Let $A$ be a Noetherian ring.
\begin{enumerate}[\rm(a)]
\item For any element $\ba\in\Um_n(A)$, there exists $\bb=(b_1,\ldots,b_n)\in\Um_n(A)$, such that $\ba\underset{A}\sim\bb$ and
$\height\big(Ab_1+\cdots+Ab_i\big)\ge i$ for all $i$. 
\item For an ideal $J\subset A[t]$, we have $\height_AL(J)\ge\height_{A[t]}J$.
\end{enumerate}
\end{lemma}

These are, respectively, Lemmas 3.4 and 3.2 in \cite[Ch. III]{Lam}.

\subsection{Polytopes and cones}\label{Polytopes} For generalities on polytopes and cones, we refer the reader to \cite[Ch. 1]{Kripo}. Here we only recall a few basic facts and conventions.

For a subset $X\subset\RR^r$, its \emph{convex hull} will be denoted by $\conv(X)$ and the \emph{affine hull} will be denoted by $\Aff(X)$. 

All our polytopes are assumed to be convex. 

The empty set is the polytope of dimension $-1$.
 
 A \emph{simplex} is a polytope of the form $\Delta=\conv(x_1,\ldots,x_n)$ such that $\dim\Delta=n-1$.

A \emph{pyramid} with \emph{apex} $v$ and \emph{base} $P$ means $\conv(v,P)$, where $v\notin\Aff(P)$.

The relative interior of a polytope $P$ will be denoted by $\inte(P)$. By convention, $\inte(P)=P$ when $P$ is a point.

Let $\RR_+$ denote the non-negative reals. For two polytopes $P\subset Q$, sharing a vertex $v$, we say that $Q$ is \emph{tangent} to P at $v$ if $\dim P=\dim Q$ and we have the equality of sets $\RR_+(P-v)=\RR_+(Q-v)$. 

A \emph{cone} will always refer to a rational, finite, pointed cone in $\RR^r$ for some $r$, i.e., a subset of the form $C=\RR_+z_1+\cdots+\RR_+z_n\subset\RR^r$ for some $z_1,\ldots,z_n\in\ZZ^r$, which contains no nontrivial subspace.

For a nonzero cone $C\subset\RR^r$, there is a rational affine hyperplane $\mathcal H\subset\RR^r\setminus\{0\}$, such that $C=\RR_+(C\cap\mathcal H)$ \cite[Proposition 1.21]{Kripo}. In this case $C\cap\mathcal H$ is a rational polytope of dimension $\dim C-1$; i.e., the vertices of $C\cap\mathcal H$ belong to $\QQ^r$.

\subsection{Affine monoids}\label{Monoids} A detailed information on the monoids of interest can be found in \cite[Ch. 2]{Kripo}. Below we give a quick review.

All our monoids are commutative, cancellative, and with unit. 

The maximal subgroup of a monoid $M$ will be denoted by $U(M)$. 

The \emph{group of differences} of a monoid $M$, also known as the \emph{Grothendieck group} of $M$, will be denoted by $\gp(M)$. Thus $M$ embeds into $\gp(M)$ and $\gp(-)$ is a left adjoint of the embedding of the category of monoids into that of abelian groups.

The additive monoid of nonnegative integers will be denoted by $\ZZ_+$. 

Unless specified otherwise, (i) we use additive notation for the operation in a monoid $M$ but switch to the multiplicative notation when $M$ is considered inside the monoid ring $R[M]$, and (ii) we will make the natural identifications $R[\ZZ_+]=R[t]$ and $R[\ZZ_+^r]=R[t_1,\ldots,t_r]$.

The submonoid of a monoid $M$, generated by elements $m_1,\ldots,m_n$, will be denoted by $\ZZ_+m_1+\cdots+\ZZ_+m_n$.

A monoid $M$ is \emph{torsion free} if $\gp(M)$ has no nonzero torsion. An \emph{affine monoid} is a finitely generated torsion free monoid. Every affine monoid $M$ isomorphically embeds into $\ZZ^r$, where $r=\rank(\gp(M))$. We put $\rank(M)=\rank(\gp(M)$.

If $M$ is affine and $U(M)=0$, then $M$ is said to be \emph{positive}.

\begin{lemma}\label{Gordan}
Assume $M$ is a monoid, $\gp(M)=\ZZ^r$, and $U(M)=0$.
\begin{enumerate}[\rm(a)]
\item \emph{(Gordan Lemma)} $M$ is affine if and only if $\RR_+M$ is a cone in $\RR^r$.
\item
If $M$ is affine then, for any ring  $R$, the monoid ring $R[M]$ admits a grading $R[M]=R\oplus R_1\oplus\cdots$ where the elements of $M$ are homogeneous.
\end{enumerate}
\end{lemma}

These are proved in Proposition 2.17(f) and Corollary 2.10(a) in \cite{Kripo}.

An affine positive monoid has the smallest generating set -- the set of indecomposable elements in $M$. It is called the \emph{Hilbert basis} of $M$ and denoted by $\Hilb(M)$.

\subsection{Normal and seminormal monoids}\label{Normal-and-seminormal} A torsion free monoid $M$ is \emph{normal} if $x\in\gp(M)$ and $nx\in M$ for some $n\in\NN$ imply $x\in M$. The \emph{normalization} of a torsion free monoid $M$ is the smallest normal submonoid $\n(M)\subset\gp(M)$, containing $M$, i.e., $\n(M)=\{x\in\gp(M)\ |\ nx\in M\ \text{for some}\ n\in\NN\}$. If $M$ is affine then $\n(M)$ is also affine.

\begin{lemma}\label{conductor}
Let $M$ be an affine positive monoid and $\gp(M)=\ZZ^r$.
\begin{enumerate}[\rm(a)]
\item
There exists $m\in M$ such that $m+\n(M)\subset M$.
\item
$M$ is normal if and only if $M=(\RR_+M)\cap\ZZ^r$.
\item
There exists a basis $n_1,\ldots,n_r\in\ZZ^r$, such that $M\subset\ZZ_+n_1+\cdots+\ZZ_+n_r$.
\item
The relative interior $\inte(\RR_+M)$ contains a basis $\{m_1,\ldots,m_r\}\subset\ZZ^r$. 
\end{enumerate}
\end{lemma}

These are the statements 2.33, 2.24, 2.17(e), and 2.74 in \cite{Kripo}, respectively.

A torsion free monoid $M$ is \emph{seminormal} if $x\in\gp(M)$ and $2x,3x\in M$ imply $x\in M$ or, equivalently, $x\in\gp(M)$ and $nx\in M$ for all $n\gg0$ imply $x\in M$. The \emph{seminormalization} $\sn(M)$ is the smallest seminormal monoid in $\gp(M)$, containing $M$. Explicitly, $\sn(M)=\{x\in\gp(M)\ |\ nx\in M\ \text{for all}\ n\gg0\}$. For $M$ affine, $\sn(M)$ is also affine.

\subsection{$\Phi$-correspondence}\label{Phi} To describe relations between submonoids of an affine positive monoid $M$ we will follow the following conventions:
\begin{enumerate}[\rm$\centerdot$]
\item
The group $\gp(M)$ will be thought of as $\ZZ^r$, where $r=\rank(M)$;
\item In $\RR^r$ there will be implicitly (sometimes explicitly) chosen a rational affine hyperplane $\mathcal H\subset\RR^r\setminus\{0\}$, such that $\RR_+M=\RR_+\big((\RR_+M)\cap\mathcal H\big)$;
\item
$\Phi(L):=(\RR_+L)\cap\mathcal H$ for a submonoid $L\subset M$ and $\Phi(m):=(\RR_+m)\cap\mathcal H$ for a nonzero element $m\in M$;
\item
For a convex subset $P\subset\Phi(M)$ we introduce the submonoid 
$$
M(P)=\{m\in M\setminus\{0\}\ |\ \Phi(m)\in P\}\cup\{0\}\subset M;
$$
\item If $M$ is normal and $P\subset\mathcal H$ is a convex subset then we introduce the submonoid
$$
M(P)=\{z\in\ZZ^r\setminus\{0\}\ |\ (\RR_+z)\cap\mathcal H\in P\}\cup\{0\}.
$$
(By Lemma \ref{conductor}(b), the last two notations are compatible when $P\subset\Phi(M)$.)
\end{enumerate}

Gordan Lemma can be rephrased as follows: a nonzero submonoid $L$ of an affine positive monoid $M$ is affine if and only if $\Phi(L)$ is a rational polytope.

For an affine positive normal monoid $M$, an element $m\in M$ is an \emph{extremal generator} if $\Phi(m)$ is a vertex of $\Phi(M)$ and $m$ is the generator of $(\RR_+m)\cap M\cong\ZZ_+$. Thus, $M$ has as many extremal generators as there are the vertices of $\Phi(M)$.

For an affine positive nonzero monoid $M$, we define the \emph{interior submonoid} and \emph{interior ideal} of $M$ by $M_*=M(\inte(\Phi(M))$ and $\inte(M)=M_*\setminus\{0\}$, respectively. Because of the convention $\inte(P)=P$ for $P$ a point (Section \ref{Polytopes}), we have $M_*=M$ when $\rank(M)=1$.

\begin{lemma}\label{localization-seminormal}
\begin{enumerate}[\rm(a)]
Let $M$ be an affine positive monoid of rank $r$.
\item
If $M$ is normal and $m\in M$ is an extremal generator then $\ZZ m+M=\ZZ m+M_0\cong\ZZ m\times M_0$ for an affine positive normal submonoid $M_0\subset\ZZ^r$ of rank $r-1$.
\item
$M$ is seminormal if and only if 
$M(F)_*=\n(M(F))_*$ for every face $F\subset\Phi(M)$ (including $\Phi(M)$ itself).
\end{enumerate}
\end{lemma}

These are, respectively, Propositions 2.32 and 2.40 in \cite{Kripo}.

\section{Pyramids and monoids}\label{Pyramidal}

\subsection{Pyramidal decomposition} Pyramidal descent is based on the following

\begin{definition}\label{pyramidal-valuation}
Let $M$ be an positive affine normal monoid and $m\in M$ be an extremal generator. A representation $M=M(\Delta)\cup M(\Gamma)$ is called a \emph{pyramidal decomposition} with \emph{vertex} $m$ if $\Delta,\Gamma\subset\Phi(M)$ are rational polytopes, such that:
\begin{enumerate}[\rm(a)]
\item $\Phi(M)=\Delta\cup\Gamma$,
\item $\Delta$ is a pyramid with base $\Delta\cap\Gamma$ and $\dim\Delta=\dim\Gamma$.
\item $\Phi(M)$  is tangent to $\Delta$ at $\Phi(m)$ (Section \ref{Polytopes}).
\end{enumerate}
\end{definition}
\noindent(Our pyramidal decompositions are called \emph{non-degenerate} in \cite{Swan}, the \emph{degenerate} ones corresponding to the case $\dim\Gamma<\dim\Delta$, i.e., when $\Gamma$ is a facet of $\Delta$.)

For a pyramidal decomposition $M=M(\Delta)\cup M(\Gamma)$, there exists a monoid homomorphism $\deg:M\to\ZZ$, such that $0\not=\deg(M(\Delta))\subset\ZZ_+$ and $\deg(M(\Gamma))\subset\ZZ_-$, where $\ZZ_-$ is the set of nonnegative integers. Such a map will be called a \emph{pyramidal degree,} associated to the given decomposition. We always have if $0\not=\deg(M(\Gamma))$.

\begin{lemma}\label{highest-degree}
Let $M$ be an affine  positive normal monoid, $M=M(\Delta)\cup M(\Gamma)$ be a pyramidal decomposition, and $\deg:M\to\ZZ$ be an associated pyramidal degree. Consider a system of linear combinations
$l_i=\sum_{j=1}^r d_{ij}n_j$, $i=1,\ldots,s$, where
\begin{enumerate}[\rm$\centerdot$]
\item $n_j\in M\setminus\{0\}$ and $d_{ij}\in\ZZ_+$ for all $i,j$,
\item $n_1$ is the vertex of the decomposition,
\item
$(d_{11},\ldots,d_{1r})>(d_{21},\ldots,d_{2r})>\ldots$ in the lexicographical order, 
\end{enumerate}
Assume $d_{ij}=a_{ij}+b_{ij}$, where $a_{ij},b_{ij}\in\ZZ_+$, $i=2,\ldots,s$ and $j=1,\ldots,r$. Then
\begin{align*}
\deg\bigg(\sum_{j=1}^r d_{1j}c_jn_1\bigg)>\deg\bigg(\sum_{j=1}^r (a_{ij}n_j+b_{ij}c_jn_1)\bigg),\quad i=2,\ldots,s,
\end{align*}
for $c_1\gg c_2\gg\ldots\gg c_r\gg0$.
\end{lemma}

\begin{proof}
We can assume $c_j>\deg(n_j)$ for $j=2,\ldots,s$. Then it is enough to achieve
\begin{align*}
\deg\bigg(\sum_{j=1}^r d_{1j}c_jn_1\bigg)>\deg\bigg(\sum_{j=1}^r d_{ij}c_jn_1\bigg),\quad i=2,\ldots,s,
\end{align*}
which is equivalent to the inequalities
\begin{align*}
\sum_{j=1}^r d_{1j}c_j>\sum_{j=1}^r d_{ij}c_j,\quad i=2,\ldots,s,
\end{align*}
known to be satisfied for $c_1\gg c_2\gg\ldots\gg c_r\gg0$.
\end{proof}

\subsection{Admissible configurations}\label{Admissible}

\begin{definition}\label{admissible}
For an affine positive normal monoid $M$ of rank $r$ and an extremal generator $m\in M$, a triple $(\mathcal H,\Delta_1,\Delta_2)$ is called an \emph{admissible configuration} if $\mathcal H\subset\RR^r\setminus\{0\}$ is a rational affine hyperplane, such that $\RR_+M=\RR_+\big((\RR_+M)\cap\mathcal H\big)$, and $\Delta_1,\Delta_2\subset\mathcal H$ are rational pyramids with apex $\Phi(m)$, satisfying the conditions:
\begin{enumerate}[\rm(a)]
\item
$\Phi(M)\subset\Delta_1\subset\Delta_2$,
\item $\Delta_1$ an $\Delta_2$ are tangent to $\Phi(M)$ at $\Phi(m)$,
\item $M(\Delta_1)=\ZZ_+m+M(F_1)$, where $F_1\subset\Delta_1$ is the facet, opposite to $\Phi(m)$,
\item $M(\Delta_2)=\ZZ_+m+M(F_2)$, where $F_2\subset\Delta_2$ is the facet, opposite to $\Phi(m)$,
\item $F_2\cap\Delta_1=\emptyset$.
\end{enumerate}
\end{definition}

In the notation above, $M(F_1)$ and $M(F_2)$ are isomorphic monoids and, consequently, so are the monoids $M(\Delta_1)$ and $M(\Delta_2)$. 

\begin{lemma}\label{admissible-exists}
An admissible configurations exists for any affine positive normal monoid $M$ with $\rank(M)\ge2$ and any extremal generator $m\in M$. 
\end{lemma}

\begin{proof}
By Lemma \ref{localization-seminormal}(a), $\ZZ m+M=\ZZ m+ M_0\cong\ZZ m\times M_0$ for an affine positive normal submonoid $M_0\subset\gp(M)$. Let $n_1,\ldots,n_{r-1}$ be a basis of $\gp(M_0)\cong\ZZ^{r-1}$, such that $M_0\subset \ZZ_+n_1+\cdots+\ZZ_+n_{r-1}$ (Lemma \ref{conductor}(c)). Then, for $k\gg0$, the following triple is an admissible configuration:
\begin{align*}
&\mathcal H=\Aff(m,n_1-(k+1)m,\ldots,n_{r-1}-(k+1)m)\subset\RR^r,\\
&\Delta_1=\big(\RR_+m+\RR_+(n_1-km)+\cdots+\RR_+(n_{r-1}-km)\big)\ \cap\ \mathcal H,\\
&\Delta_2=\conv(m,n_1-(k+1)m,\ldots,n_{r-1}-(k+1)m).
\end{align*}
\end{proof}

\section{Patching unimodular rows}\label{Patching}

We say that a commutative square of ring homomorphisms 
$$
\xymatrix{
A\ar[r]\ar[d]&A_1\ar[d]\\
A_2\ar[r]&A'
}
$$
has the \emph{Milnor patching property} for unimodular rows if for every natural number $n\ge3$ and every element $\bf\in\Um_n(A_1)$, whose image $\bf'\in\Um_n(A')$ satisfies $\bf'\underset{A'}\sim\be$, there exists $\bg\in\Um_n(A)$ with the image $\bg'\in\Um_n(A_1)$, satisfying $\bg'\underset{A_1}\sim\bf$. 

A \emph{Karoubi square} is a commutative square of rings of the following type
$$
\xymatrix{
A\ar[rr]^\rho\ar[dd]&&B\ar[dd]\\
\\
S^{-1}A\ar[rr]_{S^{-1}\rho}&&\rho(S)^{-1}B
}
$$
where $S\subset A$ is a multiplicative subset, $S$ acts regularly on $A$, $\rho(S)$ acts regularly on $B$, and the homomorphism $A/sA\to B/\rho(s)B$ is an isomorphism for every $s\in S$. A Karoubi square is always a pull-back diagram.

\begin{lemma}\label{Milnor}\emph{(Milnor Patching)} The following two types of commutative squares of ring homomorphisms have the Milnor patching property:
\begin{enumerate}[\rm(a)]
\item
$\xymatrix{
A\ar[r]\ar[d]&A_1\ar[d]^{\sigma}\\
A_2\ar[r]_{\pi}&A'
}$ a pull-back diagram, where either $\pi$ or $\sigma$ is surjective,
\item
$\xymatrix{
A\ar[rr]^\rho\ar[dd]&&B\ar[dd]\\
\\
S^{-1}A\ar[rr]_{S^{-1}\rho}&&\rho(S)^{-1}B
}$ a Karoubi square.
\end{enumerate}
\end{lemma}

This is proved in \cite[Proposition 9.1(a)]{Elrows2}, with the details for one sketched step included in \cite[Lemma 8]{Subintegral}. The basis is the equality $\E_n(A\rq{})=\E_n(A_1)\E_n(A_2)$, obvious for the case (a) and proved in \cite[Lemma 2.4]{Vorst} for the case (b). (Vorst actually shows $\E_n(A\rq{})=\E_n(A_2)\E_n(A_1)$, but the two equalities are equivalent.)

Let $A=A_0\oplus A_1\oplus\cdots$ be a graded ring and $\bf\in\Um_n(A)$. Let $\bf(0)$ denote the image of $\bf$ in $\Um_n(A_0)$ under the augmentation $A\to A_0$. In this notation we have

\begin{proposition}\label{Quillen}\emph{(Quillen Patching)} Assume $\bf(0)=\be$. Then $\bf\underset{A}\sim\be$ if and only if $\bf_\mu\underset{A_\mu}\sim\be$ for every maximal ideal $\mu\subset A_0$.
\end{proposition}
\noindent (Here $A_\mu=(A_0\setminus\mu)^{-1}A$.)

This is \cite[Corollary 7.4]{Elrows}. It is based on Suslin's $K_1$-analogue \cite[Theorem 3.1]{Suslin-Linear} of Quillen's well-known local-global patching for projective modules \cite[Theorem 1]{Quillen}.

\section{Integral and subintegral extensions}\label{Integral}

\subsection{Integral extensions} Let $M=M(\Delta)\cup M(\Gamma)$ be a pyramidal decomposition of an affine positive normal monoid $M$ with vertex $m$ and $\deg:M\to\ZZ$ be an associated pyramidal degree. For an element $f\in R[M]$, we define $\deg(f)$ as the maximum degree of the terms of $f$. Call an element $f=\sum_{j=1}^k r_jm_j\in R[M]$ \emph{monic} (with respect to $\deg$) if $r_1m_1=um^c$ for some $u\in U(R)$, $c\in\NN$, and $\deg(m_1)>\deg(m_j)$ for $j=2,\ldots,k$. 

Call a ring homomorphism $A\to B$ \emph{integral} if $B$ is integral over $\Im(A)$.

\begin{lemma}\label{integral-extension}
Let $M=M(\Delta)\cup M(\Gamma)$ be a pyramidal decomposition, $\deg:M\to\ZZ$ be an associated pyramidal degree, and $f\in R[M]$ be a monic element. Then the ring homomorphism $R[M(\Gamma)]\to R[M]/fR[M]$ is integral.
\end{lemma}

\begin{proof}
Since $M$ is normal, every element $g\in R[M]$ admits a representation $g=fg_1+f_1$, where $g_1,f_1\in R[M]$ with $\deg(f_1)<\deg(f)$. In particular, as an $R[M(\Gamma)]$-module, $R[M]/fR[M]$ is generated by the image of the subset $M_{<\deg(f)}:=\{m\in M\ |\ \deg(m)<\deg f\}$ and, therefore, by the image of $M_{\le\deg(f)}:=\{m\in M\ |\ \deg(m)\le\deg f\}$. But, according to \cite[Theorem 2.12]{Kripo}, $RM_{\le\deg(f)}$ is a finitely generated $R[M(\Gamma)]$-module.
\end{proof}

\subsection{Subintegral extensions} Call an extension of rings $A\subset B$ \emph{elementary subintegral} if $B=A[b]$ for some $b\in B$ with $b^2,b^3\in A$. A ring extension $A\subset B$ is \emph{subintegral} if it is a filtered union of elementary subintegral extensions.

\begin{theorem}\label{subintegral}
For a subintegral extension of rings $A\subset B$ and an element $\bf\in\Um_n(A)$, where $n\ge3$, we have $\bf\underset{A}\sim\be$ if and only if $\bf\underset{B}\sim\be$.
\end{theorem}

This is the main result of \cite{Subintegral}. It represents the `unimodular row' counterpart of \cite[Theorem 14.1]{Swan}. Although the statement is about arbitrary rings, the proof in \cite{Subintegral} uses monoid rings. More precisely, it uses the main result of \cite{Elrows}. A stronger result in the context of the \emph{Euler class groups} under subintegral extensions was later derived in \cite{India2}.

\section{Reduction to the interior of normal monoids}\label{Reduction}

\begin{lemma}\label{reduction}
\begin{enumerate}[\rm(a)]
\item
In order to prove Theorem \ref{main} for torsion free monoids, it is sufficient to prove it in the special case when $R$ is a local Noetherian ring and $M=N_*$ for an affine positive normal monoid $N$.
\item
In order to prove Theorem \ref{main} for a ring $R$ and an affine positive monoid $M$, it is enough to prove it for the monoid rings $R[M(F)_*]$, where $F$ runs over the faces of $\Phi(M)$ (including $\Phi(M)$ itself).
\end{enumerate}
\end{lemma}

\begin{proof}
(a) Let $L$ be a torsion free monoid. Since $L$ is the inductive limit of its affine submonoids, we can assume that $L$ is itself affine. Consider the pullback diagram
$$
\begin{aligned}
\xymatrix{
R\left[\left(L\setminus U(L)\right)\cup\{1\}\right]\ar[r]\ar[dd]&R[L]\ar[dd]^\pi\\
\\
R\ar[r]&R[U(L)]},
\end{aligned}\qquad \ker\pi=R(L\setminus U(L)).
$$

Since $U(L)$ is a free abelian group, Theorem \ref{Suslin-main} yields the transitivity of the elementary actions on $\Um_n(R[U(L)])$. Applying Lemma \ref{Milnor}(a) and using that $\left(L\setminus U(L)\right)\cup\{1\}$ is the union of its affine submonoids, we can also assume that $L$ is an affine positive monoid. Since $\sn(L)$ is a filtered union of subintegral extensions of monoids, by Theorem \ref{subintegral} we can further assume $L$ is seminormal. 

Since the elementary action on $\Um_n(R)$ is transitive for $n\ge d+2$, Lemma \ref{Gordan}(b) and Proposition \ref{Quillen} make it possible to reduce the general case to $R$ local.

Let $F_1,\ldots,F_k$ be the set of non-empty faces of the polytope $\Phi(L)$, including $\Phi(L)$ itself, indexed in such a way that $i\le j$ implies $\dim F_i\le\dim F_j$. In particular, $F_k=\Phi(L)$. 
We have the pull-back diagrams of $R$-algebras with the natural horizontal injective maps:
$$
(\mathbb D_i)\quad
\begin{aligned}
\xymatrix{
R[L(F_i)_*]\ar[rr]\ar[dd]&& R[L]\big/\sum_{j>i}R\inte(L(F_j))\ar[dd]^{\pi_i}\\
\\
R\ar[rr]&&R[L]\big/\sum_{j\ge i}R\inte(L(F_j))
},\\
\\
\ker\pi_i=R(L\cap\inte(F_i)).
\end{aligned}
$$
(Above, we assume $\sum_{j>k}R\inte(L(F_j))=0$.)

By Lemma \ref{localization-seminormal}(b), for every $i$, the ring at the upper-left corner of $(\mathbb D_i)$ is of the type $R[N_*]$, where $N$ is an affine, positive, and normal monoid. Moreover, the ring at the upper-right corner of $(\mathbb D_k)$ is $R[L]$ and that at the lower-right corner of $(\mathbb D_1)$ is $R$. Consequently, Lemma \ref{Milnor}(a) allows induction on $i$. 

\medskip\noindent(b) The claim follows from the same inductive process, based on the corresponding pull-back diagrams. 
\end{proof}

\section{Quasi-monic elements}\label{Quasimonic}

The following lemma and its proof, suggested by the referee, is a simplification of the original overly complicated version.

\begin{lemma}\label{avoidance}
Let $B$ be a Noetherian ring of dimension $d$, $\ba=(b_1,\ldots,b_n)\in\Um_n(B)$ for some $n\ge2$, and $Bb_1+By=B$. Then, for any finite family of ideals $\nu_1,\ldots,\nu_p\subset B$, there is an element $b\in Bb_2+\cdots+Bb_n$ such that, for all natural numbers $c$, one has 
$$
b_1+by^c\notin\nu_i,\qquad i=1,\ldots,p.
$$
\end{lemma}

\begin{proof}
Enlarging the $\nu_i$ one may assume they are maximal ideals. We may assume that they are distinct, hence comaximal. If $b_1\notin\nu_i$ for all $i$, take $b=0$. Else reorder and choose $q$ so that $b_1\in\nu_i$ if and only if $i\le q$. Using the Chinese Remainder Theorem, choose $z\in B$ so that $z\notin\nu_i$ if and only if $i\le q$. The restriction to $Bz$ of
$B\to\prod_{i\le q} B/\nu_i$
is still a surjective map and it has $Bzb_1$ in its kernel. So $Bzb_2+\cdots+Bzb_n$ maps onto $\prod_{i\le q} B/\nu_i$. Choose $b\in Bzb_2+\cdots+Bzb_n$ so that $b\notin\nu_i$ if $i\le q$. Notice that $by^c\notin\nu_i$ if and only if $b_1\in\nu_i$.
\end{proof}

\begin{definition}\label{tilted}
Let $R$ be a ring. An $R$-subalgebra (not necessarily Noetherian or monomial) $A\subset R[\ZZ_+^r]$ is called \emph{$t_1$-tilted} if there is a neighborhood $\Phi(t_1)\in\mathcal U\subset\Phi(\ZZ_+^r)$ such that
$m\in\ZZ_+^r\setminus\{0\}$ and $\phi(m)\in\mathcal U$ imply $m\in A$.
\end{definition}

For an element $\bg\in\Um_n(R[t_1,\ldots,t_r])$ we denote by $\bg|_{t_1=0}$ the image of $\bg$ in $\Um_n(R[t_2,\ldots,t_r])$ after substituting $0$ for $t_1$.

\begin{lemma}\label{pre-quasimonic}
Let $R$ be a local Noetherian ring of dimension $d$. Assume $A\subset R[\ZZ_+^r]$ is a $t_1$-tilted subalgebra and $\bg=(g_1,\ldots,g_n)\in\Um_n(A)$ for some  $n\ge2$, such that $\bg|_{t_1=0}\in\Um_n(R)$. 
Then there exists $\bh=(h_1,\ldots,h_n)\in\Um_n(A)$, for which $\bg\underset{A}\sim\bh$ and
$\height_{R[\ZZ_+^r]}\left(R[\ZZ_+^r]h_1+\cdots+R[\ZZ_+^r]h_i\right)\ge i$ for all $i$.
\end{lemma}

\begin{proof}
For an element $\bl=(l_1,\ldots,l_n)\in\Um_n(A)$ we introduce the condition
\begin{equation}\label{coprime-series}
R[\ZZ_+^r]l_1+R[\ZZ_+^r]t_1=\cdots=R[\ZZ_+^r]l_{n-1}+R[\ZZ_+^r]t_1=R[\ZZ_+^r].
\end{equation} 

Since $R$ is local and $\bg|_{t_1=0}\in\Um_n(R)$, by a suiatable elementary transformation over $R$ we can achieve $g_1|_{t_1=0}=\cdots=g_{n-1}|_{t_1=0}=1$. Consequently, without loss of generality, we can assume that $\bg$ satisfies (\ref{coprime-series}).

Let $0\le k<n$. Assume there exists $\bh'=(h'_1,\ldots,h'_n)\in\Um_n(A)$, for which
\begin{enumerate}[\rm$\centerdot$]
\item $\bg\underset{A}\sim\bh'$,
\item
the condition (\ref{coprime-series}) is satisfied,
\item the components satisfy the inequalities in Lemma \ref{pre-quasimonic} for $i<k$ (when $k>1$).
\end{enumerate}
 
We will induct on $k$ to achieve the height inequalities for all $i\le n-1$; there is nothing to prove for $i=n$.

Without loss of generality, $R[\ZZ_+^r]h_1'+\cdots+R[\ZZ_+^r]h_{k-1}'\not=R[\ZZ_+^r]$. Put
$$
\{\nu_1,\ldots,\nu_p\}=\begin{cases}
\text{the minimal primes over $R[\ZZ_+^r]h_1'+\cdots+R[\ZZ_+^r]h_{k-1}'$, if $k>1$},\\
\text{the minimal primes in $R[\ZZ_+^r]$, if $k=1$}.
\end{cases}
$$
Since $\bh'$ satisfies (\ref{coprime-series}), Lemma \ref{avoidance} implies the existence of an element
\begin{align*}
\tilde h= f_1h'_1+\cdots+&f_{k-1}h'_{k-1}+f_{k+1}h'_{k+1}+
\cdots+f_nh'_n,\\
&\ f_1,\ldots,f_{k-1},f_{k+1},\ldots,f_n\in\RR[\ZZ_+^r],
\end{align*} 
such that for all natural numbers $c$ one has
$$
h'_k+\tilde ht_1^c\not\in\nu_s,\qquad s=1,\ldots,p.
$$
Since  $h'_1,\ldots,h'_{k-1}\in\nu_s$ for every $s$, the element $h'=f_{k+1}h'_{k+1}+
\cdots+f_nh'_n$ satisfies
\begin{equation}\label{the-avoidance}
h'_k+h't_1^c\not\in\nu_s,\qquad s=1,\ldots,p,\qquad c\in\NN.
\end{equation}
Since $A$ is $t_1$-tilted, for $c\gg0$ we also have 
$$
h't_1^c\in Ah'_{k+1}+\cdots+Ah'_n.
$$
The last inclusion implies that, for $c\gg0$, we have
\begin{equation}\label{transformed}
\bh'\underset{A}\sim(h'_1,\ldots,h'_{k-1},h'_k+h't_1^c,h'_{k+1},\ldots,h'_n).
\end{equation}
In view of (\ref{the-avoidance}), the minimal primes over the ideal
$$
R[\ZZ_+^r]h'_1+\cdots +R[\ZZ_+^r]h'_{k-1}+R[\ZZ_+^r]\left(h'_k+h't_1^{c_j}\right)\subset R[\ZZ_+^r]
$$
are not among the $\nu_s$. In particular, the unimodular row on the right of (\ref{transformed}) satisfies (\ref{coprime-series}) and the height inequalities for $i=1,\ldots,k$.
\end{proof}

\medskip For an element $f\in R[\ZZ_+^r]$, its leading term in the lexicographical order with respect to $t_1>\ldots>t_r$ will be denoted by $L(f)$. Call an element $f\in R[\ZZ_+^r]$ \emph{quasi-monic} is $L(f)=um$ for some $u\in U(R)$ and $m\in\ZZ_+^r$.

\begin{corollary}\label{quasimonic}
Let $R$ be a local Noetherian ring of dimension $d$. Assume $A\subset R[\ZZ_+^r]$ is a $t_1$-tilted $R$-subalgebra and $\bg=(g_1,\ldots,g_n)\in\Um_n(A)$ for some $n\ge d+2$. Assume $\bg|_{t_1=0}\in\Um_n(R)$. Then there exists $\bh=(h_1,\ldots,h_n)$, such that $h_n$ is quasi-monic and $\bg\underset{A}\sim\bh$.
\end{corollary}

\begin{proof}
By Lemma \ref{pre-quasimonic}, without loss of generality we can assume $\height(I)\ge d+1$ for $I=R[\ZZ_+^r]g_1+\cdots+R[\ZZ_+^r]g_{d+1}$. Then Lemma \ref{Lam-height}(b), applied to the decreasing sequence of rings 
$R[t_1,t_2,\ldots,t_r]\supset R[t_2,\ldots,t_r]\supset\ldots\supset R[t_r]$, implies that $I$ contains a quasi-monic element $g$. Assume $L(g)=um$ for some $u\in U(R)$ and $m\in\ZZ_+^r$. Then, using that $A$ is $t_1$-tilted, we can take
$
\bh=(g_1,\ldots,g_{n-1},g_n+t_1^cg)$ with $c\gg0$. In fact, for $c\gg0$ we have $L(g_n+t_1^cg)=t_1^c L(g)$ and $t_1^cg\in Ag_1+\cdots+Ag_{d+1}$.
\end{proof}

\section{Pyramidal descent}\label{Pyramidal-descent}

\subsection{Local pyramidal descent}\label{Local-pyramidal-descent} For the rest of Section \ref{Local-pyramidal-descent}, we assume that:
\begin{enumerate}[\rm$\centerdot$]
\item
$(R,\mu)$ is a local Noetherian ring of dimension $d$,
\item $n\ge\max(d+2,3)$,
\item $M$ is a positive, affine, normal monoid with $\rank(M)\ge2$, 
\item $M=M(\Delta)\cup M(\Gamma)$ is a pyramidal decomposition with vertex $m$,
\item $(\mathcal H,\Delta_1,\Delta_2)$ is an admissible configuration for the pair $(M,m)$,
\item $\mathfrak m\subset R[M(\Gamma)]$ is the maximal ideal generated by $\left(M(\Gamma)\setminus\{1\}\right)\cup\mu$.
\end{enumerate}

Consider the following commutative diagram of rings, where the vertical arrows refer to the identity embeddings:
\begin{equation}
\label{two-pullback}
\xymatrix{
A_2\ar[dddd]\ar[rrrr]&&&&R[M]\ar[dddd]^{\bigcap}\\
&A_1\ar[ul]\ar[rr]\ar[dd]&&R[M]\ar@{=}[ur]\ar[dd]^{\bigcap}&\\
\\
&R[\ZZ_+^r]\ar[dl]\ar[rr]_{\pi_1}&&R[M(\Delta_1)]\ar[dr]&\\
R[\ZZ_+^r]\ar[rrrr]_{\pi_2}&&&&R[M(\Delta_2)]
}
\end{equation}
where:
\begin{enumerate}[\rm$\centerdot$]
\item $r=\#\Hilb(M(\Delta_1)=\#\Hilb(M(\Delta_2)$ (see Section \ref{Monoids} for Hilbert bases),
\item the $R$-algebra homomorphisms $\pi_i$ ($i=1,2$) are induced by surjective monoid homomorphisms $h_i:\ZZ_+^r\to M(\Delta_i)$, satisfying the conditions: $h_1(t_1)=h_2(t_1)=m$ and the following triples of points are collinear
$$
\Phi(m),\Phi(h_1(t_j)),\Phi(h_2(t_j))\in\mathcal H,\qquad j=2,\ldots,r,
$$
\item the inner and outer squares are pull-back diagrams,
\item the slanted arrows are the induced rings embeddings.
\end{enumerate} 

We will keep the identification $R[\ZZ_+^r]=R[t_1,\ldots,t_r]$ at the lower-left corner of the outer square, and think of $\ZZ_+^r$ at the lower-left corner of the inner square as the multiplicative monoid, generated by $t_1,t_1^{k_2}t_2,\ldots,t_1^{k_r}t_r$ for appropriate $k_2,\ldots,k_r\in\NN$ so that the slanted arrows in $(\ref{two-pullback})$ become the identity embeddings.  In particular, the upper maps in the squares are the restrictions of the respective bottom maps and so we use the same notation $\pi_i$ for them.

Denote $\ZZ_+^r(M)=h_2^{-1}(M)$. Then $A_2=R\left[\ZZ_+^r(M)+\ker\pi_2\right]\subset R[\ZZ_+^r]$, the $R$-subalgebra generated by $\ZZ_+^r(M)\cup\ker\pi_2$.

Observe that the subalgebra $A_2\subset R[\ZZ_+^r]$ is $t_1$-tilted. It is \emph{neither} a finitely generated \emph{nor} a monomial $R$-algebra as soon as $\ker\pi_2\not=0$, i.e., when $M(\Delta_2)$ or, equivalently, $M(\Delta_1)$ is not a free monoid. 

\begin{theorem}\label{descentlocal}
Assume the elementary action on $\Um_n(R[M(\Delta_1)])$ is transitive. Then $\bf_{\mathfrak m}\underset{R[M]_\mathfrak m}\sim\be$ for every $\bf\in\Um_n(R[M])$. 
\end{theorem}
\noindent(Here $R[M]_\mathfrak m=(R[M(\Gamma)]\setminus\mathfrak m)^{-1}R[M]$.)

\begin{proof}
Let $\bf=(f_1,\ldots,f_n)\in\Um_n(R[M])$. Since $\bf\underset{R[M(\Delta_1)]}\sim\be$, Lemma \ref{Milnor} implies the existence of $\bg=(g_1,\ldots,g_n)\in\Um_n(A_1)$ such that $\pi_1(\bg)\underset{R[M]}\sim\bf$. We can assume $\pi_1(\bg)=\bf$. 

Considering $\bg$ as an element of $\Um_n(A_2)$, we have $\bg|_{t_1=0}\in\Um_n(R)$. (This is where we use admissible configurations: the corresponding condition may not be satisfied over $A_1$.) By Corollary \ref{quasimonic}, without loss of generality we can assume that
\begin{equation}\label{preliminary-quasimonic}
g_n\ \text{is quasi-monic, not in}\ R.
\end{equation}

Fix a rational subsimplex $\Delta_0\subset\Phi(\ZZ_+(M))$ such that $\Phi(\ZZ_+^d)$ is tangent to $\Delta_0$ at $\Phi(t_1)$ and $\ZZ_+^r(\Delta_0)\cong\ZZ_+^r$. This can be done by choosing
$$
\Delta_0=\conv\left(\Phi(t_1),\Phi(t_1^kt_2),\ldots,\Phi(t_1^kt_d)\right).
$$
for some $k\gg0$. We have $\ZZ_+^r(\Delta_0)\subset\ZZ_+^r(M)$.

By Theorem \ref{Suslin-main}, there exists $\epsilon\in \E_n(R[\ZZ_+^r])$, whose first row is $\bg$.

For an $r$-tuple of natural numbers $(c_1,\ldots,c_r)$, consider the $R$-algebra endomorphism 
\begin{align*}
\tau=\tau(c_1,\ldots,c_r):R[\ZZ_+^r]\longrightarrow R[\ZZ_+^r],\quad \tau(t_j)=t_j+t_1^{c_j},\qquad j=1,\ldots,r.
\end{align*}

The crucial observation is that the non-unit monomials in the reduced forms of the entries of the matrix $\epsilon^{-1}\tau(\epsilon)$ align in the direction of $t_1\in\ZZ_+^d\subset\RR_+^d$ as $c_1,\ldots,c_d\to\infty$. This implies
\begin{align*}
\alpha:=\epsilon^{-1}\tau(\epsilon)\in \GL_n(R[\ZZ_+^d(\Delta_0)]\subset\GL_n(A_2)\ \ \text{for}\ \ c_1,\ldots,c_r\gg0.
\end{align*}

By Lemma \ref{Normal} and Theorem \ref{Suslin-main}, there exist $\epsilon_0\in \E_n(R[\ZZ_+^r(\Delta_0)]$ and $\alpha_0\in \GL_{n-1}(R[\ZZ_+^r(\Delta_0)])$, such that
\begin{align*}
\alpha=
\epsilon_0\begin{pmatrix}
\alpha_0&0\\
0&1\\
\end{pmatrix}.
\end{align*}
Thus, for $c_1,\ldots,c_r\gg0$, we have
\begin{align*}
(h_1,\ldots,h_{n-1},\tau(g_n)):=
\tau(\bg)\cdot
\begin{pmatrix}
\alpha_0^{-1}&0\\
0&1\\
\end{pmatrix}&=
\bg\cdot\alpha\cdot
\begin{pmatrix}
\alpha_0^{-1}&0\\
0&1\\
\end{pmatrix}\\
&=
\bg\cdot\epsilon_0\underset{A_2}\sim\bg
\in\Um_n(A_2).
\end{align*}
Consequently, for every $c_1,\ldots,c_r\gg0$, there exist $f_1',\ldots,f_{n-1}'\in R[M]$ such that 
$$
\bf\underset{R[M]}\sim(f_1',\ldots,f_{n-1}',\pi_2(\tau(g_n))\in\Um_n(R[M])
$$ 
($\pi_2$ is the map from (\ref{two-pullback})). 
Pick a pyramidal degree $\deg:M\to\ZZ$, associated to the decomposition $M=M(\Delta)\cup M(\Gamma)$. Because of (\ref{preliminary-quasimonic}), Lemma \ref{highest-degree} implies that $\pi_2(\tau(g_n))$ is monic for $c_1\gg\ldots\gg c_1\gg0$. Then, by Lemma \ref{integral-extension}, the homomorphism $R[M(\Gamma)]_{\mathfrak m}\to R[M]_{\mathfrak m}/(\pi_2(\tau(g_n))$ is integral. Since $R[M]$ is finitely generated over $R$ and $R[M(\Gamma)]_{\mathfrak m}$ is local, the quotient ring $R[M]_{\mathfrak m}/(\pi(\tau(g_n))$ is semi-local. Consequently, \cite[Corollary 7.5]{Lam} implies the equivalence $\bar\bf_{\mathfrak m}\sim\be$ over the ring $R[M]_{\mathfrak m}/(\pi_2(\tau(g_n))$, where the `bar\rq{} refers to the reduction $\mod(\pi_2(\tau(g_n))$, i.e., 
$\bar\bf_{\mathfrak m}\in\Um_{n-1}\big(R[M]_{\mathfrak m}/(\pi_2(\tau(g_n))\big)$.
By lifting the involved elementary transformation to $R[M]_{\mathfrak m}$, one derives the desired equivalence $\bf_{\mathfrak m}\underset{R[M]_{\mathfrak m}}\sim\be$; see Proposition 5.6 \cite[Ch. I]{Lam}.
\end{proof}

\subsection{Pyramidal descent and Theorem \ref{main} for torsion free monoids}\label{Global-pyramidal} We need further facts on unimodular rows and polytopes.

The following lemma is derived in the proof of \cite[Proposition 10.3]{Swan}. It represents the second use of non-finitely generated algebras (after Section \ref{Quasimonic}).

\begin{lemma}\label{pyramidal-karoubi}
Let $M$ be an affine positive normal monoid and $M=M(\Delta)\cup M(\Gamma)$ be a pyramidal decomposition. Then for a local ring $(R,\mu)$ and the maximal ideal $\mathfrak m=R\inte(M(\Gamma))+\mu\subset R[M(\Gamma)_*]$, the following diagram is a Karoubi square:
$$
\xymatrix{
R[M(\Gamma)_*]\ar[rr]\ar[dd]&&R[M_*]\ar[dd]\\
\\
R[M(\Gamma)_*]_{\mathfrak m}\ar[rr]&&R[M_*]_{\mathfrak m}
}
$$
\end{lemma}

A sequence of rational polytopes $\{P_i\}_{i=1}^\infty$ is called \emph{admissible} if, for every $i$, either $P_{i+1}=\Gamma_i$, where $P_i=\Delta_i\cup \Gamma_i$ as for the $\Phi$-images of pyramidal decompositions, or $P_i\subset P_{i+1}\subset P_1$. The following is \cite[Lemma 2.8]{Anderson}:

\begin{lemma}\label{homothety}
For a rational polytope $P_1$ and a neighborhood $\mathcal U\subset P_1$, there exists an admissible sequence of polytopes $\{P_i\}_{i=1}^\infty$ with $P_i\subset\mathcal U$ for $i\gg0$.
\end{lemma}

For an affine positive normal monoid $M$ of rank $r$ we define its \emph{complexity} as the smallest integer $\kk(M)$ for which there exists a polytope $Q\subset\Phi(M)$ of dimension $\kk(M)-1$ and points $v_1,\ldots,v_{r-\kk(M)}\in\Phi(M)$, satisfying the equality
\begin{equation}\label{multipyramid}
\Phi(M)=\conv(v_1,\ldots,v_{r-\kk(M)},Q).
\end{equation}

The condition (\ref{multipyramid}) is equivalent to the existence of a sequence of polytopes
$$
Q\subset Q_1\subset\ldots\subset Q_{r-\kk(M)}=\Phi(M),
$$
where $Q_i$ is a pyramid over $Q_{i-1}$ for each $i\ge1$. 

Observe that, for $M$ as above, $\kk(M)=0$ if and only if $\Phi(M)$ is a simplex. In particular, by the main result of \cite{Elrows}, Theorem \ref{main} is true for an affine positive normal monoid of complexity $0$. 

For the rest of Section \ref{Global-pyramidal} we assume that $R$, $d$, and $n$ are as in Theorem \ref{main} and that $R=(R,\mu)$ is local. We also assume that $M$ is an affine positive normal monoid with $\rank(M)=r$. 

Let $\Phi(M)=\conv(v_1,\ldots,v_{r-\mathsf k(M)},Q)$, as in $(\ref{multipyramid})$. For a rational polytope $P\subset Q$, we will use the notation
$$
\tilde P=\conv(v_1,\ldots,v_{r-\mathsf k(M)},P).
$$

Observe that every pyramidal decomposition $M(P)=M(\Delta)\cup M(\Gamma)$ gives rise to the pyramidal decomposition $M(\tilde P)=M(\tilde\Delta)\cap M(\tilde\Gamma)$.

The proof of the following lemma is straightforward.

\begin{lemma}\label{complexity}
Let $P\subset Q$ be a rational polytope and $m\in M(\tilde P)$ be the extremal generator, corresponding to a vertex of $P$. Assume $(\mathcal H,\Delta_1,\Delta_2)$ is an admissible configuration for the monoid $M(\tilde P)$ with respect to $m$. Then $\kk\big(M(\Delta_1)\big)=\kk(M(\tilde\Delta_1))<\kk(M)$.
\end{lemma}

\begin{proposition}\label{global-descent}\emph{(Pyramidal descent)}  Assume Theorem \ref{main} has been shown for the monoid rings of the form $R[N]$, where $N$ is an affine positive normal monoid with $\kk(N)<\kk(M)$. Then for a rational polytope $P\subset Q$, a pyramidal decomposition $M(P)=M(\Delta)\cup M(\Gamma)$, and a unimodular row $\bf\in\Um_n(R[M(\tilde P)_*]$, there exists $\bg\in\Um_n(R[M(\tilde\Gamma)_*])$, such that
$\bf\sim_{R[M(\tilde P)_*]}\bg$.
\end{proposition}

\begin{proof} For a point $z\in\Aff(\Phi(M))$ and a real number $c$,  the homothetic transformation of $\Aff(\Phi(M))$ with coefficient $c$ and centered at $z$ will be denoted by $c^z$. 

Pick any rational point $z\in\inte(\tilde\Gamma)$ and a rational number $c\in(0,1)$, sufficiently close to $1$, such that $\bf\in\Um_n\big(R[M(c^z(\tilde P))]\big)$. We have the pyramidal decomposition
$$
M\big(c^z(\tilde P)\big)=M\big(c^z(\tilde\Delta)\big)\cup M\big(c^z(\tilde\Gamma)\big).
$$
By Lemma \ref{complexity} and the assumption on monoids of complexity $<\kk(M)$, Theorem \ref{descentlocal} applies to this pyramidal decomposition, yielding the equivalence 
$$
\bf_{\mathfrak n}\sim_{R[M(c^z(\tilde P))]_{\mathfrak n}}\be,
$$
where $\mathfrak n=R(M(c^z(\tilde\Gamma))\setminus\{1\})+\mu$. We have the inclusion
\begin{align*}
{R[M(c^z(\tilde P))]_{\mathfrak n}}\subset R[M(\tilde P)_*]_{\mathfrak m},
\end{align*}
where $\mathfrak m=R(M(\tilde\Gamma)\setminus\{1\})+\mu$. In particular, Lemma \ref{pyramidal-karoubi} implies 
$\bf\underset{R[M_*]}\sim\bg$ for some $\bg\in\Um_n(R[M(\tilde\Gamma)_*])$.
\end{proof}

\begin{proof}[Proof of Theorem \ref{main} for torsion free monoids]
By Lemma \ref{reduction}(a), it is enough to show the desired transitivity for the monoid ring $R[M_*]$. We will induct on $\mathsf k(M)$.

If $\kk(M)=0$ then $\Phi(M)$ is a simplex. Consequently, $M_*$ is a filtered union of affine positive normal monoids $L$, for which the polytopes $\Phi(L)$ are simplices, and we are done by \cite{Elrows}. Assume $\kk(M)>0$ and the transitivity has been shown for the rings of the form $R[N_*]$, where $N$ is an affine positive normal monoid with $\kk(N)<\kk(M)$. Because, for every face $F\subset\Phi(N)$,  $\kk(N(F))\le\kk(N)$ and $N(F)$ is normal if $N$ is, Lemma \ref{reduction}(b) implies the transitivity also for the rings $R[N]$ with $N$ affine, positive, normal, and $\kk(N)<\kk(M)$. By Lemma \ref{homothety}(b), there exists an admissible sequence of rational polytopes $\{P_i\}_{i=1}^\infty$, such that $P_1=Q$ and $P_j$ is a rational simplex for some $j$. By Proposition \ref{global-descent}, there exists an element $\bh\in\Um_n(R[M(\tilde P_j)]_*)$ with $\bh\underset{R[M_*]}\sim\bf$. But, by the base case, $\bh\sim_{R[M(\tilde P_j)_*]}\be$.
\end{proof}

\begin{remark}\label{Old-complexity}
Let $M$ be an affine positive normal monoid and  $(\mathcal H,\Delta_1,\Delta_2)$ an admissible configuration with respect to an extremal generator $m\in M$. Not only is the polytope $\Delta_1$ closer to a simplex compared to the polytope $\Phi(M)$, making the induction in the concluding part of the proof of Theorem \ref{main} possible, but even the monoid $M(\Delta_1)$ itself is closer to a free monoid -- it splits off a free monoid summand. Based on this observation, one can define a subtler induction process where the base case is the case of polynomial algebras. This way one can avoid the use of \cite{Elrows} and base the whole argument on Suslin\rq{}s work \cite{Suslin-Linear}. The referee found a mistake in our original attempt at inducting on the complexity of monoid structures and explained how to fix it. This would also require an extension of Lemma \ref{pyramidal-karoubi}. Proposition \ref{global-descent} and the induction at the end  of the proof of Theorem \ref{main} is a variation: it requires no changes in Lemma \ref{pyramidal-karoubi} but, on the other hand, uses \cite{Elrows}. This approach is also similar to the proof of \cite[Proposition 3.2]{Gubel-Nil}.
\end{remark}

\section{Monoids with torsion}\label{Torsion}

Unless specified otherwise, a monoid in this section means a commutative and cancellative monoid, possibly with torsion, i.e., no longer is the group $\gp(M)$ assumed to be torsion free.

For a monoid $L$, let $\t(L)$ denote the the torsion subgroup of $\gp(L)$. For a monoid $L$, denote $\bar L=\Im\big(L\to\gp(L)/\t(L)\big)$. The correspondence $L\mapsto\bar L$ is a left adjoint functor for the embedding functor of the category of monoids without torsion into that of monoids.

The \emph{normalization} $\n(L)$ and  \emph{seminormalization} $\sn(L)$ of a monoid $L$ is defined in the same way as for the class of torsion free monoids in Section \ref{Normal-and-seminormal}.

For a finitely generated monoid $L$, we will identify $\gp(L)$ with $\gp(\bar L)\times\t(L)$ along some isomorphism $\gp(L)\cong\gp(\bar L)\times\t(L)$. (There is no canonical choice for such an isomorphism.) As before, we think of $\gp(\bar L)$ as $\ZZ^r$, where $r=\rank(\gp(L))$.

We need several facts, starting with an extension of Lemma \ref{localization-seminormal}(b) to monoids with torsion.

\begin{lemma}\label{seminormal-torsion}
Let $L$ be a finitely generated monoid with trivial $U(L)$ and $\FF(L)$ be the set of faces of the cone $\RR_+\bar L$, including $0$ and $\RR_+\bar L$. For every $F\in\FF(L)$, there is a subgroup $T_F\subset\t(L)$, such that:
\begin{enumerate}[\rm(a)]
\item $\sn(L)=\bigcup_{\FF(L)}\big(\inte(\n(\bar L\cap F))\times T_F\big)$;
\item $T_{\RR_+\bar L}=\t(L)$;
\item $T_{F_1}\subset T_{F_2}$ whenever $F_1\subset F_2$.
\end{enumerate}
\end{lemma}

\begin{proof}
The extension of monoids $L\subset\n(\bar L)\times\t(L)$ is \emph{integral}, i.e., every element of $\n(\bar L)\times\t(L)$ has a positive multiple in $L$. In particular, $\n(L)=\n(\bar L)\times\t(L)$. In fact, for any element $(l,t)\in\n(\bar L)\times\t(L)$, there exists $c_1\in\NN$ such that $c_1l\in\bar L$. Assume $(c_1l,s)\in L$ for some $s\in\t(L)$. Then $(c_1c_2)\cdot(l,t)=(c_1c_2l,0)\in L$ for any multiple $c_2$ of the orders of the elements $c_1t,s\in\t(L)$.

Since $\n(L)$ is a finitely generated monoid,  $\QQ[\n(L)]$ is a finitely genarated $\QQ[L]$-algebra. (Despite the use of monoid rings here, we continue using additive notation for the monoid operation.) But  $\QQ[L]\subset\QQ[\n(L)]$ is also an integral extension. Therefore, $\QQ[\n(L)]$ is module-finite over $\QQ[L]$. There is a finite generating subset of this module 
\begin{align*}
\{(l_{i1},t_{i1})-(l_{i2},t_{i2})\ |\ (l_{ij},t_{ij})\in L,\quad i=1,\ldots,k,\quad j=1,2\}\subset\n(L).
\end{align*}

We have $(\lambda,\tau)+\n(L)\subset L$, where $(\lambda,\tau)=\sum_{i=1}^k(l_{i2},t_{i2})\in L$. (In the torsion free case, we have just recovered Lemma \ref{conductor}(a).) Consequently, all sufficiently high multiples of every element $(t,l)\in\inte(\n(\bar L))\times\t(L)$ are in $(\lambda,\tau)+\n(L)$. Hence, $\inte(\n(\bar L))\times\t(L)\subset\sn(L)$.

The same argument, applied to any face $F\subset\RR_+\bar L$, yields
$$
\inte(\n(\bar L\cap F))\times T_F=\{(l,t)\in\sn(L)\ |\ l\in\bar L\cap F\}
$$
for the subgroup
$$
T_F=\{t\in\t(L)\ |\ (l,t)\in\sn(L)\ \text{for some}\ l\in\bar L\cap F\}\subset\t(L).
$$
This implies (a,b), and the part (c) follows from the inclusion $(l+l\rq{},t)\in\inte(\n(\bar L\cap F_2))\times T_{F_2}$ for any two elements 
$(l,t)\in\inte(\n(\bar L\cap F_1))\times T_{F_1}$ and $l\rq{}\in\inte(\n(\bar L\cap F_2))$.
\end{proof}

\begin{lemma}\label{unit-torsion}
In order to prove Theorem \ref{main}, it is enough to prove it when $U(M)$ is trivial. 
\end{lemma}

\begin{proof}
We can assume $M$ is finitely generated. Let $U(M)=\ZZ^s\times H$ for a finite group $H$. Since $\dim R=\dim R[H]$, the elementary action on the unimodular $n$-rows over $R[U(M)]=R[H][\ZZ^s]$ is transitive by Theorem \ref{Suslin-main}. The pullback diagram, similar to the first square in the proof of Lemma \ref{reduction}, reduces the transitivity question to the submonoid $(M\setminus U(M))\cup\{1\}\subset M$.
\end{proof}

\begin{lemma}\label{non-monomial}
Assume $R$ is a local Noetherian ring of dimension $d$, $n\ge\max(d+2,3)$, and $B\subset R[t]$ is an $R$-subalgebra, containing $t^m$ for some $m$. Then the elementary action on $\Um_n(B)$ is transitive.
\end{lemma}

\begin{proof}
For an integral extension of Noetherian rings $R_1\subset R_2$ and an ideal $I\subset R_2$, a standard commutative algebra argument yields $\height_{R_1}(I\cap R_1)\ge\height_{R_2}I$. (See, for instance, Step 2 in the proof of \cite[Lemma 6.5]{Elrows}.)

Let $\bf=(f_1,\ldots,f_n)\in\Um_n(B)$. By Lemma \ref{Lam-height}(a), we can assume $\height(Bf_1+\cdots+Bf_{d+1})\ge d+1$. Because of the  integral extension $R[t^m]\subset B$, the height inequality above and Lemma \ref{Lam-height}(b) imply $L\big((Bf_1+\cdots+Bf_{d+1})\cap R[t^m]\big)=R$. In particular, $(Bf_1+\cdots+Bf_{d+1})\cap R[t^m]$ contains a monic polynomial $g$. Without loss of generality, $\deg(g)>0$. Since $\bf\underset{B}\sim(f_1,\ldots,f_{n-1},f_n+g^c)$ and $f_n+g^c$ is monic for $c\gg0$, we can further assume that $f_n$ is monic. 

Since $R$ is Noetherian and $R[t]$ is module finite over $R[t^m]$, the algebra $B$ is also module finite over $R[t^m]$. In particular, $B$ is a finitely generated $R$-algebra and, in particular, Noetherian. The extension $B\subset R[t]$ is integral. So, every maximal ideal in $B$ lifts to $R[t]$. The extension $R\to R[t]/(f_n)$ is integral and $R$ is local, implying $R[t]/(f_n)$ is semilocal. Since there are only finitely many maximal ideals in $R[t]$, containing $f_n$, the number of maximal ideals in $B$, containing $f_n$, is also finite, i.e., $B/(f_n)$ is semi-local. By Proposition 5.6 and Corollary 7.5 in \cite[Ch. I]{Lam}, $\bf\underset{B}\sim\be$.
\end{proof}

For two monoids $L_1,L_2$ with $U(L_1)=\{0\}$, we introduce the following submonoid
$$
L_1\leftthreetimes L_2=L_1\times L_2\setminus\{(0,x)\ |\ x\in L_2,\ x\not=0\}\subset L_1\times L_2.
$$

\begin{proof}[Proof of Theorem \ref{main}]
By Lemma \ref{unit-torsion}, we can assume that $U(M)$ is trivial.

Using Lemma \ref{seminormal-torsion} and the obvious adjustment of the argument in the proof of Lemma \ref{reduction}, the problem reduces to the transitivity of the action 
\begin{equation}\label{SIX}
\begin{aligned}
\Um_n\big(R\left[M\leftthreetimes\big(\ZZ_{n_1}\times\cdots\times\ZZ_{n_p}\big)\right]\big)&\times\E_n\big(R\left[M\leftthreetimes\big(\ZZ_{n_1}\times\cdots\times\ZZ_{n_p}\big)\right]\big)\\
&\longrightarrow\Um_n\big(R\left[M\leftthreetimes\big(\ZZ_{n_1}\times\cdots\times\ZZ_{n_p}\big)\right]\big)
\end{aligned}
\end{equation}
where $R$ and $n$ are as in Theorem \ref{main}, $M$ is a torsion free (not necessarily affine) monoid with trivial $U(M)$, and $n_1,\ldots,n_p\in\NN$. 

We will prove the transitivity of the action (\ref{SIX}) by induction on $p$. 

The base case $p=0$ corresponds to torsion free monoids, already considered in Section \ref{Global-pyramidal}. Assume the transitivity holds for $p-1$. At this points we can also assume $M$ is an affine positive monoid. In this case, pulling back a grading $R[M]=R\oplus R_1\oplus\cdots$ as in Lemma \ref{Gordan}(b) along the projection $M\leftthreetimes\big(\ZZ_{n_1}\times\cdots\times\ZZ_{n_p}\big)\to M$, we obtain a grading $R\left[M\leftthreetimes\big(\ZZ_{n_1}\times\cdots\times\ZZ_{n_p}\big)\right]=R\oplus S_1\oplus\cdots$ and so, by Proposition \ref{Quillen}, $R$ can be assumed to be local.

Consider the following pull-back diagram with the vertical identity embeddings:

$$
\begin{aligned}
&\xymatrix{
\Lambda\ar[rr]\ar[ddd]&&
R\left[M\leftthreetimes\left(\ZZ_{n_1}\times\cdots\times\ZZ_{n_p}\right)\right]\ar[ddd]^{\bigcap}\\
\\
\\
R[\ZZ_+]\left[M\leftthreetimes\left(\ZZ_{n_1}\times\cdots\times\ZZ_{n_{p-1}}\right)\right]\ar[rr]_{\pi}&&
R[\ZZ_{n_p}]\left[M\leftthreetimes\left(\ZZ_{n_1}\times\cdots\times\ZZ_{n_{p-1}}\right)\right]
}
\end{aligned},
$$
where:
\begin{enumerate}[\rm$\centerdot$]
\item
$\Lambda=A+B$,
\item $A=R\left[(M\leftthreetimes\ZZ_+)\leftthreetimes\left(\ZZ_{n_1}\times\cdots\times
\ZZ_{n_{p-1}}\right)\right]$,
\item $B$ is the subalgebra $R[t^{n_p},t^{n_p+1}-t,\ldots,t^{2n_p-1}-t^{n_p-1}]\subset R[\ZZ_+]$,
\item $\pi$ is induced by $t\mapsto x$ for some generator $x\in\ZZ_{n_p}$.
\end{enumerate}

Since $\dim R[\ZZ_{n_p}]=d$, the induction hypothesis implies that the elementary action on the unimodular $n$-rows over the ring at the lower-right corner of the diagram above is transitive. Therefore, by Lemma \ref{Milnor}(a), it is enough to show that the elementary action on $\Um_n(\Lambda)$ is also transitive.

The very last pull-back diagram to be used by us is the following
\begin{align*}
\xymatrix{
A\ar[d]\ar[r]&\Lambda\ar[d]\\
R\ar[r]&B
}
\end{align*}
with the vertical surjective $R$-homomorphisms, induced by mapping all nontrivial monoid elements to $0\in R$, and the natural injective horizontal maps. Since $R$ local, the elementary action on $\Um_n(B)$ is transitive by Lemma \ref{non-monomial}. But the elementary action on $\Um_n(A)$ is also transitive by the induction hypothesis because $M\leftthreetimes\ZZ_+$ is torsion free. Thus, Lemma \ref{Milnor}(a) completes the argument. 
\end{proof}

\medskip\noindent\emph{Question.} Does Theorem \ref{main} extend to \emph{all} commutative monoids?


\bibliographystyle{plain}
\bibliography{bibliography}

\begin{thebibliography}{10}

\bibitem{Bass}
Hyman Bass.
\newblock {\em Algebraic {$K$}-theory}.
\newblock W. A. Benjamin, Inc., New York-Amsterdam, 1968.

\bibitem{Lindel}
S.~M. Bhatwadekar, H.~Lindel, and R.~A. Rao.
\newblock The {B}ass-{M}urthy question: {S}erre dimension of {L}aurent
  polynomial extensions.
\newblock {\em Invent. Math.}, 81(1):189--203, 1985.

\bibitem{Stacked}
Arne Br{\o}ndsted.
\newblock {\em An introduction to convex polytopes}, volume~90 of {\em Graduate
  Texts in Mathematics}.
\newblock Springer-Verlag, New York-Berlin, 1983.

\bibitem{Kripo}
Winfried Bruns and Joseph Gubeladze.
\newblock {\em Polytopes, rings, and $K$-theory}.
\newblock Springer Monographs in Mathematics. Springer-Verlag, New York, 2009.

\bibitem{India2}
Mrinal~Kanti Das and Md. Ali~Zinna.
\newblock On invariance of the {E}uler class groups under a subintegral base
  change.
\newblock {\em J. Algebra}, 398:131--155, 2014.

\bibitem{Anderson}
J.~Gubeladze.
\newblock The {A}nderson conjecture and a maximal class of monoids over which
  projective modules are free.
\newblock {\em Mat. Sb. (N.S.)}, 135(177)(2):169--185, 271, 1988.

\bibitem{K1Gubel}
J.~Gubeladze.
\newblock Classical algebraic {$K$}-theory of monoid algebras.
\newblock In {\em {$K$}-theory and homological algebra ({T}bilisi, 1987--88)},
  volume 1437 of {\em Lecture Notes in Math.}, pages 36--94. Springer, Berlin,
  1990.

\bibitem{Elrows}
J.~Gubeladze.
\newblock The elementary action on unimodular rows over a monoid ring.
\newblock {\em J. Algebra}, 148(1):135--161, 1992.

\bibitem{Elrows2}
J.~Gubeladze.
\newblock The elementary action on unimodular rows over a monoid ring. {II}.
\newblock {\em J. Algebra}, 155(1):171--194, 1993.

\bibitem{Gubel-Nontrivial}
J.~Gubeladze.
\newblock Nontriviality of {$SK_1(R[M])$}.
\newblock {\em J. Pure Appl. Algebra}, 104(2):169--190, 1995.

\bibitem{HHA}
J.~Gubeladze.
\newblock {$K$}-theory of affine toric varieties.
\newblock {\em Homology Homotopy Appl.}, 1:135--145, 1999.

\bibitem{Subintegral}
J.~Gubeladze.
\newblock Subintegral extensions and unimodular rows.
\newblock In {\em Geometric and combinatorial aspects of commutative algebra
  ({M}essina, 1999)}, volume 217 of {\em Lecture Notes in Pure and Appl.
  Math.}, pages 221--225. Dekker, New York, 2001.

\bibitem{Gubel-Nil}
J.~Gubeladze.
\newblock The nilpotence conjecture in {$K$}-theory of toric varieties.
\newblock {\em Invent. Math.}, 160(1):173--216, 2005.

\bibitem{K2}
J.~Gubeladze.
\newblock The {S}teinberg group of a monoid ring, nilpotence, and algorithms.
\newblock {\em J. Algebra}, 307(1):461--496, 2007.

\bibitem{India}
A.~Krishna and H.~P. Sarwar.
\newblock {\it{K}}-theory of monoid algebras and a question of {G}ubeladze.
\newblock {\em J. Inst. Math. Jussieu}, 2017.

\bibitem{Lam}
T.~Y. Lam.
\newblock {\em Serre's problem on projective modules}.
\newblock Springer Monographs in Mathematics. Springer-Verlag, Berlin, 2006.

\bibitem{Algorithm}
Reinhard~C. Laubenbacher and Cynthia~J. Woodburn.
\newblock An algorithm for the {Q}uillen-{S}uslin theorem for monoid rings.
\newblock {\em J. Pure Appl. Algebra}, 117/118:395--429, 1997.
\newblock Algorithms for algebra (Eindhoven, 1996).

\bibitem{Quillen}
D.~Quillen.
\newblock Projective modules over polynomial rings.
\newblock {\em Invent. Math.}, 36:167--171, 1976.

\bibitem{Rao}
Ravi~A. Rao.
\newblock A question of {H}. {B}ass on the cancellative nature of large
  projective modules over polynomial rings.
\newblock {\em Amer. J. Math.}, 110(4):641--657, 1988.

\bibitem{Suslin-Linear}
A.~A. Suslin.
\newblock The structure of the special linear group over rings of polynomials.
\newblock {\em Izv. Akad. Nauk SSSR Ser. Mat.}, 41(2):235--252, 477, 1977.

\bibitem{Suslin-Stability}
A.~A. Suslin.
\newblock Stability in algebraic {$K$}-theory.
\newblock In {\em Algebraic {$K$}-theory, {P}art {I} ({O}berwolfach, 1980)},
  volume 966 of {\em Lecture Notes in Math.}, pages 304--333. Springer, Berlin,
  1982.

\bibitem{Swan}
Richard~G. Swan.
\newblock Gubeladze's proof of {A}nderson's conjecture.
\newblock In {\em Azumaya algebras, actions, and modules ({B}loomington, {IN},
  1990)}, volume 124 of {\em Contemp. Math.}, pages 215--250. Amer. Math. Soc.,
  Providence, RI, 1992.

\bibitem{Tulen}
M.~S. Tulenbaev.
\newblock The {S}teinberg group of a polynomial ring.
\newblock {\em Mat. Sb. (N.S.)}, 117(159)(1):131--144, 1982.

\bibitem{vdKallen}
Wilberd van~der Kallen.
\newblock Homology stability for linear groups.
\newblock {\em Invent. Math.}, 60(3):269--295, 1980.

\bibitem{Vorst}
T.~Vorst.
\newblock The general linear group of polynomial rings over regular rings.
\newblock {\em Comm. Algebra}, 9(5):499--509, 1981.

\end{thebibliography}

\end{document}